\documentclass{article}
\usepackage[utf8]{inputenc}
\usepackage{amsmath, amsthm, amssymb, amsfonts}
\usepackage{bm}
\usepackage{dsfont}
\usepackage[utf8]{inputenc}
\usepackage{rotate}
\usepackage{tikz}
\usepackage{tikz-cd}
\usepackage[arrow,matrix,curve]{xy} 	
\usepackage{xcolor}
\usepackage{todonotes}
\usepackage{alltt} 
\usepackage{calc}

\theoremstyle{plain}
\newtheorem{theorem}{Theorem}[section]
\newtheorem{proposition}[theorem]{Proposition}
\newtheorem{lemma}[theorem]{Lemma}

\theoremstyle{definition}
\newtheorem{definition}[theorem]{Definition}

\newcommand{\norm}[1]{\left\lVert#1\right\rVert}

\title{A p-adic Poissonian Pair Correlation Concept}
\author{Christian Wei\ss{}}
\date{\today}

\begin{document}

\maketitle

\begin{abstract} The pair correlation statistic is an important concept in real uniform distribution theory. Therefore, sequences in the unit interval with (weak) Poissonian pair correlations have attracted a lot of attention in recent time. The aim of this paper is to suggest a generalization to the p-adic integers and to prove some of its main properties. In particular, connections to the theory of p-adic discrepancy theory are discussed.
\end{abstract}

\section{Introduction}

The behavior of gaps between the first $N \in \mathbb{N}$ elements of a sequence $(x_n)_{n \in \mathbb{N}} \subset [0,1]$ on a local scale can be measured by the function
$$F_{N}(s) := \frac{1}{N} \# \left\{ 1 \leq k \neq l \leq N \ : \ \left\| x_k - x_l\right\|_{\infty} \leq \frac{s}{N} \right\},$$
where $\left\| \cdot \right\|$ is the distance of a number from its nearest integer and $s \geq 0$. The sequence $(x_n)_{n \in \mathbb{N}}$ has Poissionian pair correlations if
$$ \lim_{N \to \infty} F_{N}(s) = 2s$$
for all $s \geq 0$. Despite that a random sequence drawn independently from uniform distribution generically has Poissonian pair correlations, see \cite{Mar07}, only few explicit examples of sequences possessing this property have been found, probably the most famous in \cite{BMV15}. The concept of Poissonian pair correlations was popularized by Rudnick and Sarnak in \cite{RS98} and has since then gained a high level of attention. It has been generalized in many different regards, including e.g. to higher dimensions, see \cite{HKL19}, to weak pair correlations, see \cite{NP07, HZ21} and below, to a combination of both, see \cite{Wei22a}, to compact Riemannian manifolds, see \cite{Mar20}, and to higher order correlations, see \cite{HZ23}.\\[12pt]
In this paper, we discuss yet another type of generalization, namely by replacing $[0,1]$ by its p-adic analogue, the p-adic integers $\mathbb{Z}_p$. Although \cite{Zah04} also contained the pair correlations statistic in a p-adic setting, to the best of the author's knowledge, no systematic discussion on pair correlations of sequences in $(x_n)_{n \in \mathbb{N}} \subset \mathbb{Z}_p$ can been found in the literature yet. This paper therefore aims to fill this gap and to transfer some of the most important properties of the real pair correlation function to the p-adic integers.\\[12pt]
Let $p \in \mathbb{Z}$ be a prime number. Recall that the $p$-adic absolute value is defined as follows, see e.g. \cite{Neu99}: for $a = \frac{b}{c}$ with $b,c \in \mathbb{Z} \setminus \{ 0 \}$, let $m$ be the highest possible power with $a = p^m \frac{b'}{c'}$ and $(b'c',p)=1$. Then
$$|a|_p:= p^{-m}.$$
The p-adic numbers $\mathbb{Q}_p$ are the completion of $\mathbb{Q}$ with respect to $|\cdot|_p$. The p-adic integers
$$\mathbb{Z}_p:=\left\{ x \in \mathbb{Q}_p \, : \, \left|x\right|_p \leq 1 \right\}$$
are a subring of $\mathbb{Q}_p$ and they are the closure of $\mathbb{Z}$ in the field $\mathbb{Q}_p$. Therefore, the p-adic analogue of $[0,1] \subset \mathbb{R}$ is $\mathbb{Z}_p$. Moreover it holds that
$$\mathbb{Q}_p = \bigcup_{m \geq 0} p^{-m} \mathbb{Z}_p.$$
Thus $\mathbb{Q}_p$ is countable union of copies of $\mathbb{Z}_p$ (as $\mathbb{R}$ is a countable union of copies of $[0,1]$). Finally, the ring of units $\mathbb{Z}_p^\times$ is defined as $\mathbb{Z}_p^\times:=\left\{ y \in \mathbb{Z}_p \, : \, \left|y\right|_p = 1 \right\}$.\\[12pt]
Now, we are ready to introduce p-adic Poissonian pair correlations. We will do this in greater generality by discussing weak pair correlations. Usually, see e.g. \cite{Ste18}, a sequence $(x_n)_{n \in \mathbb{N}} \in [0,1]$ is said to have weak Poissonian pair correlations for $0 \leq \alpha \leq 1$ if
\begin{align*} 
\lim_{N \to \infty} \frac{1}{N^{2-\alpha}} \# \left\{ 1 \leq i \neq j \leq N \, : \, \norm{x_i-x_j} \leq \frac{s}{N^\alpha} \right\} = 2s
\end{align*}
for all $s \geq 0$. This notion cannot be directly transferred to the p-adic integers because only discs of the volumes $p^{-k}$ with $k \in \mathbb{Z}$ exist in the p-adic numbers $\mathbb{Q}_p$. However, it was already pointed out in \cite{NP07}, that the property of having weak Poissonian pair correlations for $\alpha$ may be rewritten as
\begin{align*} 
\lim_{N \to \infty} \frac{1}{N^2} \frac{1}{\mu\left( D(0,s/N^\alpha)\right)} \# \left\{ 1 \leq i \neq j \leq N \, : \, \norm{x_i-x_j} \leq \frac{s}{N^\alpha} \right\} = 1
\end{align*}
with $\mu(\cdot)$ denoting the Lebesgue (Haar) measure on $\mathbb{R}^+$ and $D(0,r)$ the disc of radius $r>0$ centered at $0$. This equivalent definition may now be used for $\mathbb{Z}_p$: Let $\mu$ be the Haar measure on $\mathbb{Q}_p$ (as locally compact topological space) normalized with $\mu(\mathbb{Z}_p) = 1$. Next we note that for any $z \in \mathbb{Z}_p$ and $0 \leq s \leq 1$ we have 
$$D_p(z,s) := \left\{ x \, : \, \left| x - z \right|_p \leq s \right\} \subset \mathbb{Z}_p$$ 
by definition of $\mathbb{Z}_p$ and the (strong) p-adic triangle inequality $|x+y|_p \leq \max (|x|_p,|y|_p)$. As the p-adic absolute value can only take values $p^{-k}$ we get $D_p(z,s) = D_p(z,p^{-k_0})$ for the smallest $k_0$ with $p^{-k_0} \leq s$ and hence $\mu(D_p(z,s)) = \mu(D_p(z,p^{-k_0})) = p^{-k_0}.$ We may thus define
$$F_{N,\alpha,p}(s) := \frac{1}{N^2} \frac{1}{\mu\left( D_p(0,s/N^\alpha)\right)} \# \left\{ 1 \leq i \neq j \leq N \, : \, \left|x_i-x_j\right|_p \leq \frac{s}{N^\alpha} \right\}.$$
We say that $(x_n)_{n\in \mathbb{N}} \subset \mathbb{Z}_p$ has weak Poissonian pair correlations for $0 \leq \alpha \leq 1$ if 
$$\lim_{N \to \infty} F_{N,\alpha,p}(s) = 1$$ 
for all $s \geq 0$. If $\alpha = 1$, then we also just speak of Poissonian pair correlations.\\[12pt]
As will be made precise in Theorem~\ref{thm:aPPC:disc} the definition preserves the property that a uniformly distributed sequence in $\mathbb{Z}_p$ generically has Poissonian pair correlations. Before we formally discuss the notion of uniform distribution in $\mathbb{Z}_p$ next, we already give a heuristic argument for this fact similar to the one from \cite{LS20b} for the real case: consider sequence $(x_n)_{n \in \mathbb{N}} \subset \mathbb{Z}_p$ and a fixed $N \in \mathbb{N}$. For fixed $1 \leq i \leq N$ we expect $p^{-k_0} \cdot (N-1)$ of the remaining points to fulfill $\left| x_i - x_n \right|_p \leq \frac{s}{N^\alpha}$. Since there are in total $N$ points, we expect that $F_{N,\alpha,p}(s)$ is approximately 
$$\frac{1}{N^2} \cdot \frac{1}{p^{-k_0}} \cdot N \cdot (N-1) \cdot p^{-k_0}$$
and thus converges to $1$ as $N \to \infty$.\\[12pt]
The notion of uniform distribution on $\mathbb{Z}_p$ was first introduced in pioneering work \cite{Cug62}, while a nice summary of the most important properties can be found in \cite{KN74}, Chapter 5.2 : For given $k \geq 1$, a sequence $(x_n)_{n \in \mathbb{N}} \subset \mathbb{Z}_p$ is uniformly distributed of order $k$ if for every $z \in \mathbb{Z}_p$, the limit
$$\xi_k(z) := \frac{\# D_p(\alpha,1/p^k) \cap \{ x_1,\ldots,x_N \}}{N}$$
exists and is equal to $p^{-k}$. The sequence is uniformly distributed in $\mathbb{Z}_p$ if it is uniformly distributed of order $k$ for every $k \geq 1$. As a p-adic analogue to Kronecker sequences in $[0,1]$, see again \cite{KN74}. the following result goes back to \cite{Cug62}.
\begin{theorem} \label{thm:ab:ud} Let $a,b \in \mathbb{Z}_p$. The sequence $(na+b)_{n \in \mathbb{N}}$ is uniformly distributed in $\mathbb{Z}_p$ if and only if $a \in \mathbb{Z}_p^\times$.    
\end{theorem}
The notion of uniform distribution in $\mathbb{Z}_p$ can be reformulated as follows: a sequence $(x_n)_ {n \in \mathbb{N}} \subset \mathbb{Z}_p$ is uniformly distributed if and only if for every $z \in \mathbb{Z}_p$ and $k \in \mathbb{N}$ we have
$$\lim_{N \to \infty} \left| \frac{\# D_p(z,1/p^k) \cap \{ x_1,\ldots,x_N \} }{N} - \frac{1}{p^k} \right| = 0.$$
This observation suggests to introduce the notion of $p$-adic discrepancy which allows to quantify the degree of uniformity as in the real setting. The notion stems from \cite{Cug62}, see also \cite{Som22}.
\begin{definition} Let $(x_n)_{n \in \mathbb{N}} \subset \mathbb{Z}_p$ and $N \in \mathbb{N}$. Then the $p$-adic discrepancy is defined as
$$D_N(x_n) := \sup_{\alpha \in \mathbb{Z}_p, k \in \mathbb{N}} \left| \frac{\# D_p(\alpha,1/p^k) \cap \{ x_1,\ldots,x_N \} }{N} - \frac{1}{p^k} \right|.$$
\end{definition}
It is easy to show that $\tfrac{1}{N} \leq D_N(x_n) \leq 1$ for all sequences $(x_n)_{n \in \mathbb{N}} \subset \mathbb{Z}_p$ and $N \in \mathbb{N}$. Using the notion of p-adic discrepancy, Theorem~\ref{thm:ab:ud} could be quantified in \cite{Bee69}.
\begin{theorem} \label{thm:disc:kron:quant} Let $a \in \mathbb{Z}_p^\times, b \in \mathbb{Z}_p$ and $x_n := na +b$. Then
$$D_N(x_n) = \mathcal{O} \left( \frac{1}{N} \right).$$
\end{theorem}
As was already mentioned, it is well-known in the real case that a sequence of identically distributed, independent random variables $(X_n)_{n \in \mathbb{N}}$ drawn from uniform distribution on $[0,1)$ has Poissonian pair correlations almost surely. The same holds true for weak Poissonian pair correlations in the p-adic setting for any $0 \leq \alpha \leq 1$.
\begin{theorem}  \label{thm:ud:aPPC} Let $(X_n)_{n\in\mathbb{N}}$ be a sequence of independent random variables, which are uniformly distributed on $\mathbb{Z}_p$. Then for any $0 \leq \alpha \leq 1$ the sequence almost surely has $\alpha$-Poissonian pair correlations.    
\end{theorem}
The proof of Theorem~\ref{thm:ud:aPPC} has similarities to the real case but the normalization imposes additional challenges which require extra arguments.\\[12pt]
Due to a result which has been proved independently in \cite{ALP18} and \cite{GL17} for Poissonian pair correlations (i.e. $\alpha=1$) and which has been later generalized to weak pair correlations in \cite{Ste18, Coh21, Wei22a}, we know that a sequence in $[0,1]$ with $\alpha$-weak Poissonian pair correlations is uniformly distributed. This central property transfers to the p-adic integers as well.
\begin{theorem} \label{thm:aPPC:disc} Let $0 \leq \alpha \leq 1$. If $(x_n)_{n \in \mathbb{N}} \subset \mathbb{Z}_p$ has weak Poissonian pair correlations for $\alpha$, then $\lim_{N \to \infty} D_N(x_n) = 0$.
\end{theorem}
Also here, most ingredients from the proof in the real case, see \cite{GL17} and \cite{Wei22a}, can be turned into precise arguments in the p-adic setting.\\[12pt]
Finally, we discuss the role of Kronecker sequences, which are to the author's knownledge the only known examples of sequences possessing the optimal order of convergence in terms of the p-adic discrepancy. Again we obtain a result which is totally parallel of the real case, see \cite{Wei22a}.
\begin{theorem} \label{thm:Kronecker} Let $a \in \mathbb{Z}_p^\times$ be a unit and $b \in \mathbb{Z}_p$ be arbitrary. Then $(na+b)_{n \in \mathbb{N}}$ has weak Poissonian pair correlations for any $0 \leq \alpha < 1$ but does not have Poissonian pair correlations, i.e. for $\alpha = 1$
\end{theorem} 
In Section~\ref{sec:kronecker} we give two different proofs for this fact. One proof relies on the ideas developed in \cite{Wei22a}, while the other other uses an inequality from \cite{Bee72} on the p-adic discrepancy of difference sequences. 
\paragraph{Future research.} As this paper is intended as a starting point for the topic of p-adic Poissonian pair correlations, it naturally cannot cover the topic in full generality and leaves a lot of questions open. We ask here only three questions, which in the author's view would be particularly interesting to answer:
\begin{itemize}
\item While Theorem~\ref{thm:ud:aPPC} indicates that \textit{many} sequences have $\alpha$-Poissonian pair correlations, Theorem~\ref{thm:Kronecker} excludes the canonical candidate for possessing this property. How can one find explicit examples?
\item It is well-known in the real case that $\alpha$-Poissonian pair correlations of a sequence $(x_n)_{n \in \mathbb{N}}$ imply $\beta$-Poissonian pair correlations for all $0 \leq \beta \leq \alpha$, see \cite{HZ21}. Does this property also hold in the p-adic setting? This question is open for higher order pair correlations as well. The main challenge in the p-adic setting is that there is no characterization of the Poissonian pair correlations property by an integral as in \cite{HZ21}, Lemma~9.
\item Moreover, it would be natural to further generalize the notion of p-adic Poissonian pair correlation to $g$-adic numbers and the adeles. What properties can be preserved by such a generalization? For both classes, a theory of uniform distribution has already been developed, see \cite{Mei67} and the references mentioned in \cite{KN74}.
\end{itemize}

\section{Proof of Theorem~\ref{thm:ud:aPPC}}
In this section it is proven that a sequence of independent, uniformly distributed random variables in $\mathbb{Z}_p$, generically has weak Poissonian pair correlations for any $0 \leq \alpha \leq 1$. If $\alpha < 1$, the claim follows by similar arguments as in \cite{Mar07}, \cite{HKL19} and \cite{HZ21}, which all rely on the Chebyshev-inequality and the Borel-Cantelli Lemma and only some additional care needs to be taken about the normalizing factors. However, the arguments given in the mentioned references need to be significantly extended for $\alpha = 1$ . 
\begin{proof}[Proof of Theorem~\ref{thm:ud:aPPC}] For $s > 0$ let us write 
$$R_{N,\alpha,p}(s) :=  \# \left\{ 1 \leq i \neq j \leq N \, : \, \left| X_i - X_j \right|_p  \leq \frac{s}{N^\alpha} \right\}.$$
This quantity is a real-value random variable and hence we can apply classical (real) probability theory to derive 
\begin{align*}
\mathbb{E} & = \left[ R_{N,\alpha,p}(s) \right] = N (N-1) \int_{\mathbb{Z}_p}\int_{D_p(x_1,\frac{s}{N^\alpha})}1\mathrm{d}x_2\mathrm{d}x_1\\ & = N(N-1) \mu(D_p(0,s/N^\alpha))
\end{align*}
If in addition $\frac{s}{N^\alpha}\leq 1$ holds, then
\begin{align*}
    \mathrm{Var}\left[R_{N,\alpha,p}(s)\right] & = N(N-1) \mathrm{Var}\left[\mathds{1}_{\left| X_{i_1}-X_{i_2}\right|_p\leq\frac{s_l}{N^\alpha}}\right]\\
    &=N(N-1) (1-\mu(D_p(0,s/N^\alpha)))\mu(D_p(0,s/N^\alpha))\\
    &\leq N^2\mu(D_p(0,s/N^\alpha).
\end{align*}
At first we consider the case $\alpha < 1$. Let $\varepsilon > 0$. Using Chebyshev's inequality, we obtain
$$P[|F_{N,\alpha,p}(s) - E[F_{N,\alpha,p}(s)]| \geq \varepsilon] \leq \frac{1}{\varepsilon^2} \textrm{Var}\left[F_{N,\alpha,p}\right] \leq \frac{1}{N^2 \mu(D_p(0,s/N^\alpha))}.$$
Then $\mu(D_p(0,s/N^\alpha)) \geq \frac{1}{p} s/N^\alpha$ yields
\begin{align*}
    \sum_{N=1}^\infty & P[|F_{N,\alpha,p}(s) - E[F_{N,\alpha,p}(s)]| \geq \varepsilon] \leq \frac{1}{\varepsilon^2} \sum_{N=1}^\infty \frac{1}{N^2 \mu(D_p(0,s/N^\alpha))}\\ 
    & \leq \frac{p}{\varepsilon^2 \cdot s} \sum_{N=1}^\infty \frac{1}{N^{2-\alpha}} < \infty.
\end{align*}
By the first Borel-Cantelli lemma we have
$$P[\lim_{N \to \infty} |F_{N,\alpha,p}(s) - E[F_{N,\alpha,p}(s)]| \geq \varepsilon] = 0.$$
Hence 
$$\lim_{N \to \infty} F_{N,\alpha,p}(s) = \lim_{N \to \infty} E[F_{N,\alpha,p}(s)] = 1$$
holds almost surely. This finishes the proof for $\alpha < 1$.\\[12pt]
In the case $\alpha = 1$, choose some $\gamma > 1$ and let $N_M = \lceil M^\gamma \rceil$ for $M \in \mathbb{N}$. Then the claim follows for $N_M$ instead of $N$ by essentially the same argument as in the case $\alpha < 1$. Thus, the desired convergence holds for $N_M$ instead of $N$. In order to obtain the result for $N$, we use a similar argument as e.g. in \cite{HKL19} but need to amend it at some points.\\[12pt]
First we note that $\frac{1}{p} s/M^\gamma \leq \mu(D_p(0,s/M^\gamma)) \leq s/M^\gamma$ implies
$$\frac{1}{p} \frac{(M+1)^\gamma}{M^\gamma} \frac{\mu(D_p(0,s/M^\gamma))}{\mu(D_p(0,s/(M+1)^\gamma))} \leq p \frac{(M+1)^\gamma}{M^\gamma}$$
and therefore the exponents $k,l$ with $\mu(D_p(0,s/M^\gamma))) = p^{-k}$ and $\mu(D_p(0,s/(M+1)^\gamma))) = p^{-l}$ may at most distinguish by $1$. If this happens we call such an $N_M$ special. Moreover for all $k \in \mathbb{N}$ there exists an $N_0 \in \mathbb{N}$ with the following property for all $M \geq N_0$: if $\mu(D_p(0,s/M^{\gamma})) = p \mu(D_p(0,s/(M+1)^\gamma))$ then  $\mu(D_p(0,s/(M+j)^\gamma)) = \mu(D_p(0,s/(M+1)^\gamma))$ for all $1 \leq j \leq k$. This is true because
$$- \sum_{j=0}^k\log \left( \frac{s}{(M+j)^\gamma}\right) \leq \log(p)$$
if $M$ is large enough.\\[12pt] 
For a moment, we now exclude all non-special values from the sequence $N_M$ and obtain a new sequence $N_M^*$ with $\mu(D_p(0,s/N_{M}^{*\gamma})) = p \mu(D_p(0,s/N_{M+1}^{*\gamma}))$ for all $M \in \mathbb{N}$ and  $N_{M+1}^* / N_M^* \to p$ for $M \to \infty$. As the restriction to $N_M^*$ preserves the convergence property, it follows that
\begin{align} \label{eq:lim}
\lim_{M \to \infty} \frac{\# \left\{ 1 \leq i \neq j \leq N \, : \, \left| X_i - X_j \right|_p  \leq \frac{s}{N_{M}^{*\gamma}} \right\} }{\# \left\{ 1 \leq i \neq j \leq N \, : \, \left| X_i - X_j \right|_p  \leq \frac{s}{N_{M+1}^{*\gamma}} \right\} } = p
\end{align}
as well.\\[12pt]
Now consider $N$ with $N_M \leq N \leq N_{M+1}$. Then
\begin{align} \label{ineq:F_N}
\begin{split}
& N_M^{2} \mu(D_p(0,s/N_{M+1}) F_{N_M,1,p}\left(\frac{N_M}{N_{M+1}} s\right)\\
& \leq N^2 \mu(D_p(0,s/N))F_{N,1,p}(s)\\
& \leq N_{M+1}^{2} \mu(D_p(0,s/N_M) F_{N_{M+1},1,p}\left(\frac{N_{M+1}}{N_{M}} s\right).
\end{split}
\end{align}
At first we consider all $N$ with $N_M$ being special. We either have $D_p(0,s/N^\gamma)) = D_p(0,s/N_M^\gamma)$ or $D_p(0,s/N^\gamma) = D_p(0,s/N_{M+1}^\gamma)$. If the first holds true, then 
\begin{align*}
N_M^2 & \mu(D_p(0,s/N_{M+1}^{2}) F_{N_M,1,p}\left(\frac{N_M}{N_{M+1}} s\right)\\ 
&= \frac{1}{p} \# \left\{ 1 \leq i \neq j \leq N \, : \, \left| X_i - X_j \right|_p  \leq \frac{s}{N_{M+1}^{\gamma}} \right\}
\end{align*}
Dividing both sides by $N^2\mu(D_p(0,s/N^\gamma)) = N^2 \mu(D_p(0,s/N_M^\gamma)$ and observing $N/N_M \to 1$ as $M \to \infty$ and \eqref{eq:lim}, then we see that $1$ is asymptotically a lower bound for $F_{N,1,p}(s)$. If on the other hand $\mu(D_p(0,s/N^\gamma) = \mu(D_p(0,s/N_{M+1}^\gamma)$, then
\begin{align*}
N_M^2 & \mu(D_p(0,s/N_{M+1}^{*2}) F_{N_M,1,p}\left(\frac{N_M}{N_{M+1}} s\right)\\ 
&= \# \left\{ 1 \leq i \neq j \leq N \, : \, \left| X_i - X_j \right|_p  \leq \frac{s}{N_{M+1}^{*\gamma}} \right\}
\end{align*}
Dividing by $N^2\mu(D_p(0,s/N^\gamma)) = N^2 \mu(D_p(0,s/N_{M+1}^\gamma)$ and observing $N_{M+1}/N \to 1$, it follows as well that $1$ is an asymptotic lower bound.\\[12pt]
If $N_M$ is not special, then \eqref{ineq:F_N} and $N_M/N_{M+1} \to 1$, directly yield $1$ as an asymptotic lower bound for $F_{N,1,p}(s)$. We therefore know, that $1$ is a lower bound for $\lim_{N \to \infty} F_{N,1,p}(s)$, because we have covered all possible cases. Regarding the upper bound we can argue similarly and obtain for it $1$. Thus, 
$$\lim_{N \to \infty} F_{N,1,p}(s) = 1$$
almost surely, i.e., the claim also holds true for $\alpha = 1$.
\end{proof}
\paragraph{Realization of a random uniformly distributed sequence in $\mathbb{Z}_p$.} One might ask, how Theorem~\ref{thm:ud:aPPC} can be checked empirically on a computer. Therefore, we now shortly explain how a sequence of random p-adic integers can be algorithmically realized: As in the real case, random number generation can only be made up to a given precision. In the p-adic setting this is the number of terms involved in the Laurent series of $z \in \mathbb{Z}_p$, see \cite{Neu99} for details. We need to assume that this number $K$ has been fixed. The standard approach for generating a uniform distribution in $\mathbb{Z}_p$ then works as follows: For $k=0,\ldots,K$ draw $a_k \in \{0,\ldots,p-1\}$ from a discrete uniform distribution. The random number $z_i$ is given as $z_i = \sum_{k=0}^K a_k p^k$. It can be immediately seen that $\xi_k(z)$ is equal to $p^{-k}$ for all $k \leq K$ as is the desired precision.

\section{Proof of Theorem~\ref{thm:aPPC:disc}}
In order to prove that weak Poissonian pair correlations for $0 \leq \alpha \leq 1$ imply convergence of the p-adic discrepancy to $0$, we show the following stronger assertion.
\begin{lemma} \label{lem:aPPC:disc}
Suppose there exists a function $F: \mathbb{N} \times \mathbb{N} \to \mathbb{R}$ which is monotonically increasing in its first argument, and which satisfies
\begin{align*} \max_{s=1,\ldots,K} & \left| \frac{N^{-\alpha}}{\mu(D_p(0,s/N^\alpha))} \# \left\{ 1\leq l \neq m \leq N \, : \, |x_l-x_m|_p < \frac{s}{N^\alpha} \right\} - N^{2-\alpha}) \right|\\ & < F(K,N)
\end{align*}
for some $0 < \alpha \leq 1$ and all $K \leq N/2$. Then there exists an integer $N_0>0$ such that for all $N \in \mathbb{N}$ with $N \geq N_0$, and arbitrary $K$ satisfying
$$\frac{1}{2}N^{2/5 \alpha} \leq K \leq N^{2/5 \alpha}$$
we have
$$ND_N^*(x_n) \leq 5 \max \left(N^{1-1/5\alpha},\sqrt{N^\alpha \cdot F(K^2,N)} \right)=:H(N,K).$$
\end{lemma}
The proof relies on the method developed in \cite{GL17} and is in large parts parallel to it. The additional ingredient in our proof is to transfer the p-adic setting to the real one. Thereby it is necessary to take track of constants and exponents.
\begin{proof} Assume to the contrary that the claim is wrong, i.e. $ND_N^* > H(N,K)$ for infinitely many $N,K$. Then there exists $1 < N_1 < N_2 < \ldots$ and corresponding $K_1,K_2, \ldots$ as well as $z_1,z_2, \ldots$ and $k_1,k_2,\ldots$ such that without loss of generality we have
\begin{align} \label{ineq:disc_p}
\# \{ 1 \leq n \leq N_j \, : \, x_n \in D_p(z_j,p^{-k_j})\} - N_j p^{-k_j} > H(N_j,K_j)
\end{align}
(the opposite inequality can be treated very similarly). To simplify notation we write $N=N_j, k=k_j, K=K_j$ and $z = z_j$ for fixed $j$. Let 
$$\mathcal{H}_L := \# \left\{ 1 \leq l \neq m \leq N \, : \, |x_l-x_m|_p \leq \frac{KL}{N^\alpha}\right\}.$$
Furthermore we set
$$A_i := \left\{ 1 \leq n \leq N \, : \, |z-x_n|_p \in \left[ i \cdot \frac{K}{N^{\alpha}}, (i+1) \cdot \frac{K}{N^\alpha} \right) \right\}.$$
for $i=0,1,\ldots,\lfloor N^\alpha/K\rfloor -1$ and 
$$A_{\lfloor N^\alpha/K \rfloor} := \left\{ 1 \leq n \leq N \, : \, |z-x_n|_p \in \left[ \left\lfloor \frac{N^\alpha}{K}  \right\rfloor\cdot \frac{K}{N^\alpha} , 1 \right) \right\}.$$
If $x_l \in A_i$ and $x_m \in A_{i+j}$ with $j \in \{0,\ldots,L-1\},$ then
$$|x_l-x_m|_p \leq \max \left( |z-x_l|_p,|z-x_m|_p\right) \leq \frac{LK}{N^\alpha}$$
and thus the pair $(x_l,x_m)$ contributes to $\mathcal{H}_L$. We can hence deduce that
\begin{align} \label{ineq:HL} \mathcal{H}_L \geq \sum_{i=0}^{\lfloor N^\alpha/K \rfloor} (A_i(A_i-1)+2A_i(A_{i+1}+\ldots+A_{i+L-1})).
\end{align}
By almost verbatim the same algebraic manipulations of \eqref{ineq:disc_p} and \eqref{ineq:HL} as in \cite{GL17}, we get
$$\frac{N^{-\alpha}}{N^{2-\alpha}\mu(D_p(0,LK/N^\alpha))} \mathcal{H}_L \geq \frac{2}{K+1}Z_K - \frac{1}{N\mu(D_p(0,LK/N^\alpha))}$$
with
\begin{align*}Z_K & \geq \frac{1}{N^2 \mu(D_p(0,K^2/N^\alpha))} \left( \frac{K^2(Np^{-k}+H)^2}{K+\lfloor N^\alpha / K \rfloor} + \frac{K^2(N(1-p^{-k})-H)^2}{\lfloor N^\alpha / K \rfloor - K - \lfloor N^\alpha p^{-k} / K \rfloor}\right) \\
& \geq \frac{1}{K^2N^{2-\alpha}} \left( \frac{K^2(Np^{-k}+H)^2}{K+\lfloor N^\alpha / K \rfloor} + \frac{K^2(N(1-p^{-k})-H)^2}{\lfloor N^\alpha / K \rfloor - K - \lfloor N^\alpha p^{-k} / K \rfloor}\right),
\end{align*}
where we used $\mu(D_p(0,K/N^\alpha) \leq K/N^\alpha$. As in \cite{Wei22a} it follows that
$$Z_K \geq K \left( 1 + \frac{H^2}{2N^2}\right) > \frac{K}{2} \left( 1 + \frac{H^2}{2N^2}\right).$$
Thus, 
$$\frac{1}{N^{2-\alpha}} F(K^2,N) + 1 \geq 1 + \frac{H^2}{2N^2} - \frac{3}{2} \frac{1}{K} - \frac{1}{N\mu(D_p(0,K/N^\alpha))}$$
and $\mu(D_p(0,K/N^\alpha) \geq \frac{1}{p} K/N^\alpha$ yields
$$\frac{1}{N^{2-\alpha}} F(K^2,N) \geq \frac{H^2}{2N^2} - \frac{3}{2} \frac{1}{K} - \frac{p}{K} N^{\alpha-1}.$$
For $N$ large enough (depending on $\alpha$ and $p$) we hence obtain
$$H^2 \leq 2N^\alpha F(K^2,N) + \frac{4N^2}{K}.$$
Then $K \geq \tfrac{1}{2} N^{\frac{2}{5}\alpha}$ implies $N^2/K < 2N^{2(1-\tfrac{1}{5}\alpha)}$ and therefore by the definition of $H$ we have
\begin{align*}
    H^2 & \leq 2N^\alpha F(K^2,N) + \frac{4N^2}{K} \\
    & < 6 \max\left(\frac{N^2}{K}, N^\alpha F(K^2,N)\right) \\
    & < 12 \max \left( N^{2(1-\tfrac{1}{5}\alpha)}, N^\alpha F(K^2,N)\right) < H^2,
\end{align*}
which is a contradiction.
\end{proof}
Theorem~\ref{thm:aPPC:disc} is an immediate consequence of Proposition~\ref{prop:disc:aPPC}.
\begin{proof}[Proof of Theorem~\ref{thm:aPPC:disc}] Let $\varepsilon > 0$ be arbitrary. Since $(x_n)_{n \in \mathbb{N}}$ has weak Poissonian pair correlations for $\alpha$ it holds that
$$\left| \frac{N^{-\alpha}}{\mu(D_p(0,s/N^\alpha))} \# \left\{ 1\leq l \neq m \leq N \, : \, |x_l-x_m|_p < \frac{s}{N^\alpha} \right\} - N^{2-\alpha}) \right| \leq \varepsilon N^{2-\alpha}$$
for $N \geq N_0$. We set $F(K,N):=\varepsilon N^{2-\alpha}$ and choose $K$ with $\frac{1}{2}N^{\tfrac{2}{5}\alpha} \leq K \leq N^{\tfrac{2}{5}\alpha}$. From Lemma~\ref{lem:aPPC:disc} it follows that
$$D_N^* \leq \frac{5}{N} \max \left( N^{1-\tfrac{1}{5}\alpha}, N^\alpha N^{2-\alpha} \varepsilon \right) = 5 \sqrt{\varepsilon},$$
which implies the claim.    
\end{proof}

\section{Proof of Theorem~\ref{thm:Kronecker}} \label{sec:kronecker}
Finally, we give two different proofs of the fact that Kronecker sequences possess weak Poissonian pair correlations for all $0 \leq \alpha < 1$. The first is analogous to \cite{Wei22a} and relies on the following proposition.
\begin{proposition} \label{prop:disc:aPPC}
If $(x_n)_{n \in \mathbb{N}} \in \mathbb{Z}_p$ has discrepancy of order $D_N(x_n) = O(N^{-\varepsilon})$ for some $1 \geq \varepsilon > 0$, then $(x_n)_{n \in \mathbb{N}}$ has $\alpha$-Poissonian pair correlations for all $0 < \alpha < \epsilon$.
\end{proposition}
\begin{proof} Since $D_N(x_n) = \mathcal{O}(N^{-\varepsilon})$, we obtain that
\begin{align*} \frac{1}{N^2} & \frac{1}{\mu\left( B(0,s/N^\alpha)\right)} \# \left\{ 1 \leq i \neq j \leq N \, | \, \left|x_i-x_j\right|_p \leq \frac{s}{N^\alpha} \right\} \\
& = \frac{1}{N^2} \frac{1}{\mu\left( B(0,s/N^\alpha)\right)} N \cdot \left( N \cdot \mu\left( B(0,s/N^\alpha)\right) + N \cdot \mathcal{O}(N^{-\varepsilon}) \right).\\
& = 1 + \frac{\mathcal{O}(N^{-\varepsilon})}{B(0,s/N^\alpha)}.\end{align*}
As $\frac{1}{p} \frac{s}{N^\alpha} < p^{-k_0} \leq \frac{s}{N^\alpha}$, the claim follows.
\end{proof}
From Proposition~\ref{prop:disc:aPPC} and Theorem~\ref{thm:disc:kron:quant} we can conclude that for $a \in \mathbb{Z}_p^\times, b \in \mathbb{Z}_p$, the sequence $(na+b)_{n \in \mathbb{N}}$ has $\alpha$-Poissonian pair correlations for all $0 < \alpha <1$ as in the real case. To prove this part of Theorem~\ref{thm:Kronecker}, we can alternatively use the following Theorem from \cite{Bee72}.
\begin{theorem} \label{thm:beer} Let $(x_n)_{n \in \mathbb{N}} \subset \mathbb{Z}_p$ have discrepancy $D_N(x_n)$. Furthermore, let $(y_n)_{n = 1}^{N^2}$ consist of the elements $x_i - x_j$ for $i,j = 1, \ldots, N$ in any ordering and let $E_N^2(y_n)$ denote its discrepancy. Then
$$D_N^2 \leq E_{N^2} \leq D_N.$$
\end{theorem}
Now let $s \geq 0$ be arbitrary. Without loss of generality we may assume that the first $2N$ elements of $y_n$ are the differences $x_i - x_i = 0$. From Theorem~\ref{thm:beer} we can then deduce 
\begin{align*}
    & \left| \frac{\# D_p(0,s/N^\alpha) \cap \left\{y_{2N+1},\ldots,y_{N^2} \right\} }{N^2} - \mu(D_p(0,s/N^\alpha)) \right|\\
    & \leq \left| \frac{\# D_p(0,s/N^\alpha) \cap \left\{y_{1},\ldots,y_{N^2} \right\}}{N^2} - \mu(D_p(0,s/N^\alpha)) \right| + \frac{2}{N}\\
    & \leq E_{N^2} + \frac{2}{N} \leq D_N + \frac{2}{N}.
\end{align*}
If $0 < \alpha < 1$, then $F_{N,\alpha,p}(s) \to 1$ as desired. However, $(na+b)_{n \in \mathbb{N}}$ can never have Poissonian pair correlations, i.e. for $\alpha = 1$. This is in parallel to the real case, where it is easy to show from the continued fraction expansion of $z \in \mathbb{R} \setminus \mathbb{Q}$ that the Kronecker sequence $(\left\{nz\right\})_{n \in \mathbb{N}} \subset [0,1]$, where $\{ \cdot \}$ denotes the fractional part, does not have Poissonian pair correlations.\\[12pt]
As $|n a + b - (m a + b)|_p =  |n - m|_p$, it suffices to consider the sequence $(n)_{n \in \mathbb{N}}$. Let $s<1$, then $p^{-k_0} \leq s/N$ implies $p^{k_0} > N$. Thus $|i-j|_p > s/N$ for all $i,j \leq N$. This proves the final claim of Theorem~\ref{thm:Kronecker}.


\bibliographystyle{alpha}
\bibdata{literatur}
\bibliography{literatur}

\newcommand{\etalchar}[1]{$^{#1}$}
\begin{thebibliography}{EBMV15}

\bibitem[ALP18]{ALP18}
C.~Aistleitner, T.~Lachmann, and F.~Pausinger.
\newblock Pair correlations and equidistribution.
\newblock {\em Journal of Number Theory}, 182:206--220, 2018.

\bibitem[Bee69]{Bee69}
S.~Beer.
\newblock {\"U}ber die {D}iskrepanz von {F}olgen in bewerteten {K}\"orpern.
\newblock {\em Manuscripta mathematica}, 1:201--210, 1969.

\bibitem[Bee72]{Bee72}
S.~Beer.
\newblock Die {D}iskrepanz von {D}ifferenzenfolgen im p-adischen.
\newblock {\em Monatshefte f\"ur Mathematik}, 76:289--294, 1972.

\bibitem[Coh21]{Coh21}
A.~Cohen.
\newblock Poissonian correlation of higher order differences.
\newblock {\em Journal of Number Theory}, 229:463--486, 2021.

\bibitem[Cug62]{Cug62}
M.~Cugiani.
\newblock Successioni uniformemente distribuite nei domini p-adici.
\newblock {\em Istituto Lombardo. Accademia di Scienze e Lettere Rendiconti A},
  92:351--372, 1962.

\bibitem[EBMV15]{BMV15}
D.~El-Baz, J.~Marklof, and I.~Vinogradov.
\newblock The two-point correlation function of the fractional parts of
  $\sqrt{n}$ is {P}oisson.
\newblock {\em Proceeding of the AMS}, 143 (7):2815--2828, 2015.

\bibitem[GL17]{GL17}
S.~Grepstad and G.~Larcher.
\newblock On pair correlation and discrepancy.
\newblock {\em Archiv der Mathematik}, 109:143--149, 2017.

\bibitem[HKL{\etalchar{+}}19]{HKL19}
A.~Hinrichs, L.~Kaltenb\"ock, G.~Larcher, W.~Stockinger, and M.~Ulrich.
\newblock On a multi-dimensional {P}oissonian pair correlation concept and
  uniform distribution.
\newblock {\em Monatshefte f\"ur Mathematik}, 190:333--352, 2019.

\bibitem[HZ21]{HZ21}
M.~Hauke and A.~Zafeiropoulos.
\newblock Weak poissonian correlations.
\newblock {\em arXiv:2112.11813}, 2021.

\bibitem[HZ23]{HZ23}
M.~Hauke and A.~Zafeiropoulos.
\newblock Poissonian correlations of higher orders.
\newblock {\em Journal of Number Theory}, 243:202--240, 2023.

\bibitem[KN74]{KN74}
L.~Kuipers and H.~Niederreiter.
\newblock {\em Uniform distribution of sequences}.
\newblock John Wiley \& Sons, New York, 1974.

\bibitem[LS20]{LS20b}
G.~Larcher and W.~Stockinger.
\newblock 7. on pair correlation of sequences.
\newblock In D.~Bilyk, J.~Dick, and F.~Pillichshamme, editors, {\em Discrepancy
  Theory}, pages 133--146. De Gruyter, Berlin, Boston, 2020.

\bibitem[Mar07]{Mar07}
J.~Marklof.
\newblock Distribution modulo one and {R}atner's theorem.
\newblock In A.~Granville and Z.~Rudnick, editors, {\em Equidistribution in
  Number Theory, An Introduction. NATO Science Series}, volume 237. Springer,
  Dordrecht, 2007.

\bibitem[Mar19]{Mar20}
J.~Marklof.
\newblock Pair correlation and equistribution on manifolds.
\newblock {\em Monatshefte f\"ur Mathematik}, 2020:279--294, 2019.

\bibitem[Mei67]{Mei67}
H.~G. Meijer.
\newblock Uniform distribution of g-adic numbers.
\newblock {\em Indagationes Mathematicae}, 29:535--546, 1967.

\bibitem[Neu99]{Neu99}
J.~Neukirch.
\newblock {\em Algebraic Number Theory}.
\newblock Springer, 1999.

\bibitem[NP07]{NP07}
R.~Nair and M.~Policott.
\newblock Pair correlations of sequences in higher dimensions.
\newblock {\em Israel Journal of Mathematics}, 157:219--238, 2007.

\bibitem[RS98]{RS98}
Z.~Rudnick and P.~Sarnak.
\newblock The pair correlation function of fractional parts of polynomials.
\newblock {\em Communication in Mathematical Physics}, 194:61--70, 1998.

\bibitem[Som22]{Som22}
N.~Somasunderam.
\newblock A leveque-type inequality on the ring of p -adic integers.
\newblock {\em Internatonal Journal of Number Theory}, 18 (3):655--671, 2022.

\bibitem[Ste18]{Ste18}
S.~Steinerberger.
\newblock Poissonian pair correlation and discrepancy.
\newblock {\em Indagationes Mathematicae}, 29 (5):1167--1178, 2018.

\bibitem[Wei22]{Wei22a}
C.~Wei\ss{}.
\newblock Some connections between discrepancy, finite gap properties and pair
  correlations.
\newblock {\em Monatshefte f\"ur Mathematik}, 199:909--927, 2022.

\bibitem[Zah04]{Zah04}
A.~Zaharescu.
\newblock Pair correlation of squares in $p$-adic fields.
\newblock {\em Canadian Journal of Mathematics}, 55 (2):432--448, 2004.

\end{thebibliography}

\textsc{Ruhr West University of Applied Sciences, Duisburger Str. 100, D-45479 M\"ulheim an der Ruhr,} \texttt{christian.weiss@hs-ruhrwest.de}

\end{document}